\documentclass[1p]{elsarticle}			
\usepackage[utf8]{inputenc}
\usepackage{microtype}
\usepackage{amsmath, amssymb, amsfonts, amsthm, mathtools}
\usepackage{MnSymbol}
\usepackage{graphicx}
\usepackage[hidelinks]{hyperref} 
\hypersetup{
            colorlinks = true, 
            urlcolor   = blue, 
            linkcolor  = blue, 
            citecolor  = red  }

\usepackage{caption}
\usepackage{subfig}  
\usepackage{keyval2e}  
\usepackage{everysel,ragged2e} 
\newtheorem{theorem}{Theorem}

\newtheorem{lemma}{Lemma}

\theoremstyle{definition}

\newtheorem{definition}{Definition}
\newtheorem{example}{Example}

\author[UoW]{Jennifer Seberry \fnref{fn1}}

\author[Nick]{N.A. Balonin \fnref{fn2}}

\fntext[fn1]{School of Computer Science and Software Engineering, Faculty of Engineering and Information Science, University of Wollongong, NSW 2522, Australia. Email: \url{j.seberry@uow.edu.au}}

\fntext[fn2]{Saint Petersburg State University of Aerospace Instrumentation,
67, B. Morskaia St., 190000, St. Petersburg, Russian Federation. Email: \url{korbendfs@mail.ru}}

\address[UoW]{University of Wollongong, NSW, Australia}

\address[Nick]{Saint Petersburg State University of Aerospace Instrumentation, St. Petersburg, Russian Federation}

\title{Equivalence of the Existence of Hadamard Matrices and Cretan$(4t-1,2)$-Mersenne Matrices}

\begin{document}

\begin{abstract}
We study orthogonal matrices whose elements have moduli $\leq 1$. This paper shows that the existence of two such families of matrices is equivalent.
Specifically we show that the existence of an Hadamard matrix of order $4t$ is equivalent to the existence of a 2-level Cretan-Mersenne matrix of order $4t-1$.
\end{abstract}

\maketitle

\textbf{Keywords}: \textit{Hadamard matrices; orthogonal matrices; symmetric balanced incomplete block designs (SBIBD); Cretan matrices; Mersenne matrices; 05B20.}

\section{Introduction}\label{sec:introduction}

\indent An application in image processing led to the search for orthogonal matrices, all of whose elements have modulus $\leq 1$ and which have  maximal or high determinant.
  
$Cretan$ matrices were first discussed during a conference in Crete in 2014 by N. A. Balonin, M. B. Sergeev and colleagues of the Saint Petersburg State University of Aerospace Instrumentation but were well known using \textit{Heritage names}, \cite{BMS02,BM06,BMS01,BMS03}. This paper follows closely the joint work of N. A. Balonin, Jennifer Seberry and M. B. Sergeev \cite{BNSJ14,BNSJ14a,BNSJSM14}.

In this and future papers we use some \textit{names, definitions, notations}
differently to how  they have been have in the past \cite{BM06}. This we hope, will cause less confusion, bring our nomenclature closer to common usage, conform for mathematical purists and clarify the similarities and differences between some matrices. We have chosen to use the word level, instead of value for the entries of a Cretan matrix, to conform to earlier writings \cite{BM06, BMS01, BMS03}.

We know of no references where the \textbf{Co-Existence of Hadamard Matrices and Cretan$(4t-1,2)$-Mersenne Matrices Theorem} is stated.

\subsection{Preliminary Definitions}

The absolute value of the determinant of any matrix is not altered by 1) interchanging any two rows, 2) interchanging any two columns, and/or 3) multiplying any row/or column by $-1$. These equivalence operations are called \textit{Hadamard equivalence operations}. So the absolute value of the determinant of any matrix is not altered by the use of Hadamard equivalence operation.

Although it is not the definition used by purists we use orthogonal matrix as below. Write $I_n$ for the identity matrix of order $n$ and let $\omega $ be a constant.
When a matrix $S$ is written in the following form
\newcommand{\BigFig}[1]{\parbox{16pt}{\Huge #1}}
\newcommand{\BigB}{\BigFig{B}}
\[S = \begin{bmatrix}
x      & \sigma   & \dots & \sigma  \\
\sigma  \\
\vdots & & \BigB\\
\sigma 
\end{bmatrix}\] 
$B$ is said to be the \textit{core of } $S$ and the $\sigma $'s are the \textit{borders} of $B$ in $S$. The variable $x$ is set as zero for conference matrices, $\sigma$ for Hadamard matrices and $a$ for Fermat matrices.

\begin{definition} [\textbf{Orthogonal Matrix, Hadamard Matrices, Cretan Matrix}] \label{had-orthog-cretan} 
An \textit{orthogonal matrix}, $S= (s_{ij})$ of order $n$, is square and has the modulus of all its entries $\leq 1$, and satisfies $SS^{\top} = \omega I_n$ for $n = 1, 2, 4t$.

An \textit{Hadamard matrix} of order has entries $\pm 1$  and satisfies
$HH^{\top} = nI_n$ for $n$ = 1, 2, $4t$, $t > 0$ an integer. Any Hadamard matrix can be put into \textit{normalized form}, that is having the first row and column all plus 1s using Hadamard equivalence operations: that is it can be written with a core.

A \textit{Cretan matrix}, $S$,  of order $v$ has entries with modulus $\leq 1$ and at least one element in  each row and column  must be 1, and which satisfies $SS^{\top} = \omega I_v$. A $Cretan(n; \tau; \omega )$ matrix, or $CM(n; \tau ; \omega)$ has $\tau $ levels or values for its entries.
\end{definition}

In this work we will only use orthogonal to refer to matrices comprising real elements with modulus $\leq 1 $, where at least one entry in each row and column must be one. 

Hadamard matrices and weighing matrices are the best known of these orthogonal matrices. We refer to \cite{BNSJ14a,SY92,JH1893,JSW72} for more definitions. We recall Barba \cite{Barba33} showed that for matrices whose entries have modulus $\leq 1$, $B$, of order $n$
\begin{equation}\label{eq:B}
\det{B} \leq \sqrt{2n-1}(n-1)^{\frac{n-1}{2}} \textnormal{ or asymptotically } \approx 0.858(n)^{\frac{n}{2}}\,.
\end{equation}

Wotjas \cite{Wojtas64} showed that for matrices whose entries have modulus $\leq 1$, $B$, of order $n \equiv 2 \pmod{4}$ we have

\begin{equation*}\label{eq:C}  
\det{B} \leq 2(n-1)(n-2)^{\frac{n-2}{2}} \textnormal{ or asymptotically } \approx 0.736(n)^{\frac{n}{2}}\,.
\end{equation*}
More details of Cretan matrices can be found in \cite{NJ15a}. 

When Hadamard introduced his famous inequality \cite{JH1893}, for matrices with moduli $\leq 1$ it was noticed that Hadamard  matrices, which are orthogonal and with entries $\pm 1$ satisfied the equality of Hadamard's inequality.

However the inequality applies to other matrices and orders with entries on the unit disk, these have been named ``Fermat", Mersenne and Euler matrices for different congruence classes $\pmod{4}$, \cite{BNSJ14a,NJ15a}.

We emphasize: in general a Cretan($n,\tau $) or $CM(n,\tau)$ or $S$, has $\tau$ levels, they are made from \textit{$\tau $-variable orthogonal matrices} by replacing the variables by appropriate real numbers with moduli $\leq 1$, where at least one entry in each row and column is 1. After this stage, $\tau $-variable orthogonal matrices and Cretan($n,\tau $), are used, \textit{loosely} to denote one-the-other.

For 2-level Cretan matrices we will denote the levels/values by $x ,y$   where $0 \leq |y| \leq x = 1$. We also use the notations \textit{Cretan(v)}, \textit{Cretan(v)-SBIBD} and \textit{Cretan-SBIBD} for Cretan matrices of 2-levels and order $v$ constructed using $SBIBD$s.

\begin{definition}[\textbf{SBIBD and Incidence Matrix}]\label{def:incidence-matrix-SBIBD} 
For the purposes of this paper we will consider an $SBIBD(v, k, \lambda)$, $B$, to be a $v \times v$ matrix, with entries $0$ and $1$, $k$ ones per row and column, and the inner product of distinct pairs of rows and/or columns to be $\lambda$. This is called the \textit{incidence matrix} of the SBIBD. For these matrices $\lambda(v-1) = k(k-1)$.
\end{definition}

For every $SBIBD(v, k, \lambda)$ there is a complementary $SBIBD(v, v-k, v-2k + \lambda)$. One can be made from the other by interchanging the $0$'s of one with the $1$'s of the other. The usual use $SBIBD$ convention that $v >2k$ and $k >2\lambda$ is followed.

A combinatorial trick allows us to say that any matrix of order $n$ satisfying $AA^{\top} =aI +bJ$ has determinant $\left(\sqrt{a + nb}\right)a^{\frac{n-1}{2}}$.

Writing $A$ for the incidence matrix of the $SBIBD(v, k, \lambda)$, which has entries 0 and 1, and $B=2A -J$, $J$ the matrix of all ones, for the $\pm 1$ it forms we have:

\begin{multline}\label{eq:A}
AA^{\top} = (k-\lambda)I + \lambda J \hspace{0.5cm} AJ=kJ \textnormal{~and~} \\BB^{\top} =2(k - \lambda )I + (v- 2(k-\lambda ))J \hspace{0.5cm} BJ=(2k-v)J.
\end{multline}
Thus 
 \[det(A) = k(k-\lambda)^{\frac{v-1}{2}} \quad \textnormal{and} \quad  det(B) = \sqrt{k^2 + (v-k)^2)}(2(v-k)g^{\frac{v-1}{2}}\,.\]
 
We now define our important concepts the \textit{orthogonality equation}, the \textit{radius equation(s)}, the \textit{characteristic equation(s)} and the \textit{weight} of our matrices.

\begin{definition}[\textbf{Orthogonality equation, radius equation(s), characteristic equation(s), weight}]\label{def:or-rad-char}
Consider the matrix $S= (s_{ij})$ comprising  the variables $x$ and $y$.

The \textit{matrix orthogonality equation} 
  \[ S^{\top }S = SS^{\top} = \omega I_n \]
yields two types of equations: the $n$ equations which arise from taking the inner product of each row/column with itself (which leads to the diagonal elements of $\omega I_n$ being $\omega$) are called \textit{radius equation(s)}, $g(x,y)=\omega$, and the $n^2 -n$ equations, $f(x,y)=0$, which arise from taking inner products of distinct rows of $S$ (which leads to the zero off diagonal elements of $\omega I_n$)  are called \textit{characteristic equation(s)}. The \textit{orthogonality equation} is $\sum_{j=1}^n s_{ij}^2 =  \omega$. $\omega$ is called the \textit{weight} of $S$. \qed
\end{definition}

\begin{example}\label{eg:4a-b}
We consider the 2-variable $S$ matrix given by 
\[  S = \begin{bmatrix}
                x &  y &  y &  y & y \\
                y &  x &  y &  y & y \\
                y &  y &  x &  y & y \\
                y &  y &  y &  x & y \\
                y &  y &  y &  y & x 
        \end{bmatrix}\,. 
\]

By definition, in order to become an orthogonal matrix, it must satisfy the orthogonality equation, $SS^{\top}= \omega I$, the radius and the characteristic equations, so we have 
\[ x^2 + 4y^2 = \omega, \qquad 2xy +3 y^2 = 0\,.\]

To make a Cretan(5;2;$\frac{10}{3}$) we force $x =1$, (since we require that at least one entry per row/column is 1), and the characteristic equation gives $y=-\frac{2}{3}$. Hence $\omega = 3{\frac{1}{3}}$. The determinant is $(\frac{10}{3})^{\frac{5}{2}}$ = 20.286. Thus we now have an $S = Cretan(5;2;{\frac{10}{3}};20.286).$ \qed
\end{example}

\subsection{Notation Transitions}\label{notation-transition}

In transiting from one mother tongue to another (Russian to English and English to Russian) and from previous to newer usage, some words reoccur: we need a shorthand. To simplify references we note:

\begin{table}[h]
\begin{center}
\begin{tabular}{ll|l}
\textbf{Heritage Usage}& Cretan Matrix & References \\
\hline
Fermat         & Cretan(4t+1)    & \cite{BM06,BNSJ14a,NJ15,BMS01} \\
Hadamard       & Cretan(4t)      & \cite{BNSJ14,BMS03,BNSJ14a,JH1893,SY92}\\
Mersenne       & Cretan(4t-1)    & \cite{BMS02,BM06,BMS03,BNSJ14a,Sorrento,Singapore}\\ 
Euler          & Cretan(4t-2)    & \cite{BMS02,BMS01,BMS03}.
\end{tabular}
\caption{Cretan and Heritage Names}
\label{table:cretan-and-heritage-names}
\end{center}
\end{table}

\begin{table}[h]
\begin{center}
\begin{tabular}{ll|l}
\textbf{Usage}& Dual Usage & Heritage name  \\
\hline
CM(4t+1)       & Core of $CM(4t+2)$    & ``Fermat" matrices\\
CM(4t)         & Core of $CM(4t+1)$    & Hadamard matrix \\
CM(4t-1)       & Core of $CM(4t)$      & Mersenne matrix \\
CM(4t-2)       & Core of $CM(4t-1)$    & Euler matrix    \\
CM(4t-3)       & Core of $CM(4t-2)$    & ``Fermat" matrices\\
\hline
\end{tabular}
\caption{Cretan and Core Names}
\label{table:cretan-and-core-names}
\end{center}
\end{table}

\section{Preliminary Results and Hadamard Mathematical Foundations}\label{sec:Had}

As an historical note we point out that J. A. Todd's \cite{Todd33} article showing the relationship of $SBIBD(4t-1,2t-1,t-1)$ and Hadamard matrices of order $4t$ appeared in the same issue of the \textit{Journal of Mathematics and Physics} as the famous paper by R. E. A. C. Paley \cite{Paley33} using Legendre symbols to construct orthogonal matrices.

\begin{example} \label{example:b}
An $SBIBD(7,4,2) = B$ and its complementary $SBIBD(7,3,1)$ can be written with incidence matrices: they are still complementary if permutations of rows/columns are applied to one and other permutations of rows/columns to the other. This is because if $P$ and $Q$ are permutation matrices $PBQ$ is equivalent to the $SBIBD(7,4,2)$.

\begin{multline*}
SBIBD(7,4,2) = \begin{bmatrix}
                   1 & 1 & 1 & 0 & 1 & 0 & 0\\
                   0 & 1 & 1 & 1 & 0 & 1 & 0 \\
                   0 & 0 & 1 & 1 & 1 & 0 & 1 \\
                   1 & 0 & 0 & 1 & 1 & 1 & 0 \\
                   0 & 1 & 0 & 0 & 1 & 1 & 1 \\
                   1 & 0 & 1 & 0 & 0 & 1 & 1 \\
                   1 & 1 & 0 & 1 & 0 & 0 & 1
\end{bmatrix}\\
SBIBD(7,3,1) = \begin{bmatrix}
                   0 & 1 & 1 & 0 & 1 & 0 & 0\\
                   0 & 0 & 1 & 1 & 0 & 1 & 0 \\
                   0 & 0 & 0 & 1 & 1 & 0 & 1 \\
                   1 & 0 & 0 & 0 & 1 & 1 & 0 \\
                   0 & 1 & 0 & 0 & 0 & 1 & 1 \\
                   1 & 0 & 1 & 0 & 0 & 0 & 1 \\
                   1 & 1 & 0 & 1 & 0 & 0 & 0 
\end{bmatrix}
\end{multline*}

The second incidence matrix is still a complement of the incidence matrix of the first $SBIBD$ even after permutations of its rows and/or columns have been performed.
          
To prepare to make Cretan matrices from $SBIBD(4t-1,2t-1,t-1)$ we first put the two $SBIBD$ in variable forms: or $SBIBD(7,4,2)$ \text{circ}($x,x,x,y,x,y,y$) and for the complementary $SBIBD(7,3,1)$, circ($y,x,x,y,x,y,y$), for 
\allowdisplaybreaks
\begin{multline*}
SBIBD(7,4,2) = \begin{bmatrix}
                   x & x & x & y & x & y & y\\
                   y & x & x & x & y & x & y \\
                   y & y & x & x & x & y & x \\
                   x & y & y & x & x & x & y \\
                   y & x & y & y & x & x & x \\
                   x & y & x & y & y & x & x \\
                   x & x & y & x & y & y & x
\end{bmatrix}\\
SBIBD(7,3,1) = \begin{bmatrix}
                   y & x & x & y & x & y & y\\
                   y & y & x & x & y & x & y \\
                   y & y & y & x & x & y & x \\
                   x & y & y & y & x & x & y \\
                   y & x & y & y & y & x & x \\
                   x & y & x & y & y & y & x \\
                   x & x & y & x & y & y & y 
\end{bmatrix}
\end{multline*} 

We could just have exchanged $x$ and $y$ in the definition \text{circ}($x,x,x,y,x,y,y$) to get the second $SBIBD$, but we wanted to illustrate that there are many other possibilities for the first row of the second $SBIBD$.

\begin{figure}[h] 
  \centering
  \subfloat[][the principal solution]{\includegraphics[width=0.4\textwidth]{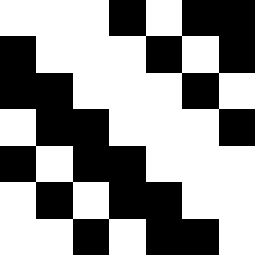}} \qquad \qquad
  \subfloat[][the complementary  solution]{\includegraphics[width=0.4\textwidth]{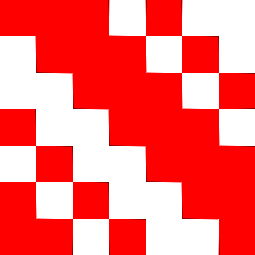}}\\
  \caption{Orthogonal matrices for order $7$: Balonin-Sergeev Family}
  \label{fig:S7}
\end{figure}  

Consider the 2-variable $SBIBD(7,4,2)$: it has characteristic equation $2x^{2} + 4xy+ y^{2} = 0$, and radius equation $\omega = 4x^2 + 3y^2$. $\det(S) = \omega^{\frac{7}{2}}$. This is an $CM(7,2)$. To make a Cretan matrix we now set $x=1$ and solve the characteristic equation to find $y$ in terms of $x$. This value/level is then used to give the Cretan$(7;2;5.0294)$. Thus the principal solution has 
\[x = 1,\hspace{1cm}  y = - 2 + \sqrt{2},\; \hspace{1cm}  \omega = 4x^{2} + 3y^{2} =  5.0294\;.\]

The $SBIBD(7,3,1)$ with characteristic equation $x^{2} + 4xy+ 2y^{2} = 0$, and radius equation $\omega = 3x^2 + 4y^2$, $\det(S) = \omega^{\frac{7}{2}}$,  (smaller than above values replacing $y$), has a feasible solution for the characteristic equation. Proceeding as before we obtain a second solution

\[x = 1,\hspace{1cm} y = \frac{-2+\sqrt{2}}{2}\;, \hspace{1cm} \omega = 3x^{2} + 4y^{2} = 3.3431\;.\]
Thus it gives a Cretan$(7;2;3.3431)$ matrix.

The two determinants are $285.31$ and $69.319$ respectively. 

Loosely we write $Cretan(7;2:5.2904)-SBIBD$, (2-variable orthogonal $SBIBD$ $(7,4,2)$), and obtain the following 2-level Cretan-Mersenne matrix:

\[ \begin{bmatrix*}[c]
                   1 & 1 & 1 &-2+\sqrt{2} & 1 &-2+\sqrt{2} &-2+\sqrt{2}\\
                  -2+\sqrt{2} & 1 & 1 & 1 &-2+\sqrt{2} & 1 &-2+\sqrt{2} \\
                  -2+\sqrt{2} &-2+\sqrt{2} & 1 & 1 & 1 &-2+\sqrt{2} & 1 \\
                   1 &-2+\sqrt{2} &-2+\sqrt{2} & 1 & 1 & 1 &-2+\sqrt{2} \\
                  -2+\sqrt{2} & 1 &-2+\sqrt{2} &-2+\sqrt{2} & 1 & 1 & 1 \\
                   1 &-2+\sqrt{2} & 1 &-2+\sqrt{2} &-2+\sqrt{2} & 1 & 1 \\
                   1 & 1 &-2+\sqrt{2} & 1 &-2+\sqrt{2} &-2+\sqrt{2} & 1 
\end{bmatrix*} 
\]\qed
\end{example}

\section{Cretan-SBIBD(v,2) Theorem: The Cretan 2-level Matrices from SBIBD Theorem: }\label{sec:main SBIBD-CRETAN}

Although our main theorem appeared as a corollary in \cite{NJ15a} it is in fact worthy of being a theorem in its own right. The paper \cite{NJ15a} gives many relevant details.  We write

\begin{theorem}[\textbf{The Cretan-SBIBD(v,2) Theorem: The Cretan 2-level Matrices from SBIBD Theorem}] \label{principal-2-var-theorem}
Whenever there exists an $SBIBD(v,k,\lambda)$ there exists a  2-level Cretan-SBIBD$(v,2)$, or $S$ or $CM$, as follows,
\begin{itemize}
\item  Cretan$\left(v;2;kx^2 + (v-k)y^2; \textnormal{ y,x;determinant} \right);$ 
\item  Cretan$\left(v;2;(v-k)x^2 +ky^2; \textnormal{ y,x;determinant}\right)$.
\end{itemize}
We have used the notation Cretan(order;$\tau ;\omega;$ \textnormal{ y,x;determinant}).
 \qed
\end{theorem}

In all these $Cretan-Hadamard$ cases (but not in all cases) the Balonin-Sergeev-Cretan$(4t-1,2)$ matrix with higher determinant comes from $SBIBD(4t-1,2t,t)$ while the $SBIBD(4t-1,2t-1,t-1)$ gives a Cretan$(4t-1,2)$ matrix with smaller determinant. These examples have been given as they may give circulant SBIBD when other matrices do not necessarily do so.

\section{Main Equivalence Theorem}\label{sec:main-equiv-theorem}

\begin{lemma}[\textbf{Hadamard to SBIBD Lemma}] \label{lem:H toSBIBD} 
Whenever there exists an Hadamard matrix of order $4t$ there exists an $SBIBD(4t-1,2t-1,t-1)$.
\end{lemma}
\begin{proof}
Let $H$ be the Hadamard matrix of order $4t$. We use the Hadamard equivalence operations to make the normalized Hadamard matrix $G$ with the first row and column all $+1$'s.

It now has the form
\[G = \begin{bmatrix}
    1 & 1  & \dots & 1 \\
    1 \\
    \vdots & & \BigB\\
    1
\end{bmatrix}\]
where $B$ is a $\pm 1$ matrix of order $4t-1$ containing $2t$ $-1$'s and $2t-1$ 1's per row and column. $B$ satisfies $BB^{\top} = 4tI - J$. We form $A$ by replacing the $-1$ elements of $B$ by zero, that is $A =\frac{1}{2}(J+B)$. Then $A$ satisfies $AA^{\top} = tI + (t-1)J$ and is the incidence matrix of an $SBIBD(4t-1,2t-1,t-1)$ as required.
\end{proof}

\begin{lemma}[\textbf{SBIBD to Cretan$(4t-1,2)$ Lemma}] \label{lem:SBIBD to Cretan(4t-1)}
Whenever there exists an $SBIBD(4t-1,2t-1,t-1)$, there exists an Cretan-Mersenne 2-level matrix.
\end{lemma}

\begin{proof}
We take the $SBIBD$ and replace the zeros by $x$ and the ones by $y$. Call this matrix $C$.

We choose $x=1$ and $y = \frac{-t + \sqrt{t}}{t-1}$.
We first show this satisfies the orthogonality equation. The inner product of any row with itself is $2t + (2t-1)\left(\frac{-t + \sqrt{t}}{t-1}\right)^2=\omega$ which is a constant given $t$. Thus we have the radius equation. Using Hadamard equivalence operations we permute the columns of $C$ until we have
\begin{gather*}
\overbrace{y,y,\dots,y,y,~~y,y,\dots, \;y,\;y}^{2t-1} \qquad \overbrace{x,x,\dots,x,x,~~x,x,\dots,x,x}^{2t}\\
\underbrace{y,y,\dots,y,y}_{t-1} ~~\underbrace{x,x,\dots,x,x}_{t} \qquad \underbrace{y,y,\dots,y,y}_{t} ~~ \underbrace{x,x,\dots,x,x}_{t}
\end{gather*}
 Then the inner product of rows $i$ and $j$ is
\[(t-1)y^2+2txy+tx^2= (t-1)\left (\frac{-t + \sqrt{t}}{t-1}\right )^2 + 2t\left(\frac{-t + \sqrt{t}}{t-1}\right) + t =0\,.\]
Hence $CC^{\top} =\omega I$ and we have the matrix orthogonality equation. Thus we have formed the required 2-level Cretan-Mersenne matrix.
\end{proof}

\begin{lemma} [\textbf{Cretan$(4t-1,2)$ to Hadamard Lemma}] \label{lem:Cretan(4t-1) to Had}  
Whenever there exists an $SBIBD(4t-1,2t-1,t-1)$, there exists an Hadamard matrix of order $4t$ .
\end{lemma}

\begin{proof}
We take the Cretan$(4t-1,2)$ matrix and replace the entries which are not $1$ by zero.
This gives the incidence matrix $A$ of the $SBIBD(4t-1,2t-1,t-1)$. By definition of $SBIBD$ the inner product of any pair of distinct columns is $\lambda = t-1$ so it satisfies $AA^{\top} = tI + (t-1)J$, $J$ the matrix of all $1$'s. We note that it satisfies $AA^{\top} = tI + (t-1)J$.  It has a total of $4t-1$ entries, $2t-1$ 1's and $2t$ zeros in each row and column.  We now form $B = 2A - J$, of order $4t-1$ and note $BB^{\top} = 4tI  -J$. Bordering $B$ with a row and column of $1$'s gives the required Hadamard matrix.
\end{proof}

\begin{theorem} [\textbf{Co-Existence of Hadamard Matrices and Cretan$(4t-1,2)$\\-Mersenne Matrices Theorem}] \label{th:equivalence} 
The existence of an Hadamard matrix of order $4t$ is equivalent to the existence of a Cretan$(4t-1,2)$-Mersenne matrix.
\end{theorem}

\begin{proof}
We have shown that and Hadamard matrix of order $4t$ can be used to form an $SBIBD(4t-1,2t-1,t-1)$ using Lemma \ref{lem:H toSBIBD} as its core.

Lemma \ref{lem:SBIBD to Cretan(4t-1)} was used to show that whenever there exists an $SBIBD(4t-1,2t-1,t-1)$, there exists an Cretan-Mersenne 2-level matrix.

This shows that an Hadamard matrix can always be used to obtain a Cretan-Mersenne 2-level matrix.

To prove the converse we use the Cretan$(4t-1,2)$ to Hadamard Lemma \ref{lem:Cretan(4t-1) to Had} to show how a 2-level Cretan$(4t-1,2)$ matrix can always be used to form an Hadamard matrix.
Whenever there exists an $SBIBD(4t-1,2t-1,t-1)$, there exists an Hadamard matrix of order $4t$ .
This shows that a Cretan-Mersenne 2-level matrix can always be used to obtain a Cretan-Mersenne 2-level matrix.

Thus we have shown the existence of one gives the existence of the other. This completes the proof.
\end{proof}

{\bf Conjecture:} Since Hadamard matrices are conjectured to exist for all orders $4t$, $t>0$ an integer, 2-level Cretan$(4t-1,2)$-Mersenne matrices are conjectured to exist for all orders $4t-1$, $t>0$ an integer.

\section{Conclusions}

Cretan matrices are a very new area of study. They have many research lines open: what is the minimum number of variables that can be used; what are the determinants that can be found for Cretan($n;\tau$) matrices; why do the congruence classes of the orders make such a difference to the proliferation of Cretan matrices for a given order; find the Cretan matrix with maximum and minimum determinant for a given order; can one be found with fewer levels? 

We conjecture that $\omega \approxeq v $ will give unusual conditions.\qed

\section{Acknowledgements}

 The authors also wish to sincerely thank Mr Max Norden, BBMgt(C.S.U.), for his work preparing the content and LaTeX version of this article.

We acknowledge use of \url{http://www.mathscinet.ru} and \url{ http://www.wolframalpha.com} sites for the number and symbol calculations in this paper.

\bibliographystyle{elsarticle-num}
\bibliography{Equiv_Had_Cretan_Mersenne_refs}

\end{document}